\documentclass[11pt]{amsart}
\usepackage[hmarginratio=1:1]{geometry}              
\usepackage{graphicx}
\usepackage{amssymb}
\usepackage{epstopdf}
\usepackage{amsthm} 
\usepackage{marvosym}
\usepackage{enumerate}
\usepackage[all]{xy}
\usepackage[linktocpage]{hyperref}
\usepackage{float}

\newtheorem{thm}{Theorem}
\numberwithin{thm}{section}

\newtheorem{lemma}[thm]{Lemma}

\newtheorem{proposition}[thm]{Proposition}

\newtheorem{question}[thm]{Question}

\theoremstyle{remark}
\newtheorem{rmk}[thm]{Remark}

\theoremstyle{definition}
\newtheorem{definition}[thm]{Definition}

\newcommand{\overl}[1]{\overline{#1}}

\newcommand{\R}{\mathbb{R}}

\newcommand{\C}{\mathbb{C}}

\newcommand{\Z}{\mathbb{Z}}

\newcommand{\Hy}{\mathbb{H}}
\newcommand{\HE}{\text{HE}}

\newcommand{\Aut}{\text{Aut}}
\newcommand{\Out}{\text{Out}}
\newcommand{\Mod}{\text{Mod}}

\newcommand{\Isom}{\text{Isom}}

\newcommand{\SL}{\text{SL}}

\newcommand{\G}{\Gamma}
\newcommand{\Sp}{\text{Sp}}

\title{On K\"ahler extensions of abelian groups}
\author{Corey Bregman and Letao Zhang \vspace{-2ex}}
\begin{document}
\maketitle
\begin{abstract}We show that any K\"ahler extension of a finitely generated abelian group by a surface group of genus $g\geq 2$ is virtually a product.  Conversely, we prove that any homomorphism of an even rank, finitely generated abelian group into the genus $g$ mapping class group with finite image gives rise to a K\"ahler extension.  The main tools come from surface topology and known restrictions on K\"ahler groups.
\end{abstract}
\section{Introduction}

Let $G$ be a finitely presented group.  If $G\cong \pi_1(X)$ for some compact K\"ahler manifold $X$, we say that $G$ is a \emph{K\"ahler group}.  In this paper we examine whether certain extensions of compact surface groups by abelian groups can arise as the fundamental groups of compact K\"{a}hler manifolds.  For $g\geq1$, let $\Sigma_g$ denote the orientable compact surface of genus $g$. We denote its fundamental group by $S_g.$ Every $\Sigma_g$ admits a complex structure, and in dimension 2 the K\"ahler condition is equivalent to being closed and orientable, hence every $S_g$ is a K\"ahler group. Another class of examples of K\"ahler groups come from even-dimensional tori, arising as $\C^n/\Lambda$ where $\Lambda\cong \Z^{2n}$ is a lattice.  Therefore, $\Z^{2n}$ is a K\"ahler group for all $n$, and when $n=1$, the quotient $\C/\Lambda$ is a surface of genus $g=1$.  We refer the reader to \cite{ABCKT96} for an excellent introduction to K\"ahler groups. 

Our main question concerns extensions of finitely generated abelian groups by $S_g$.  
\begin{question} Let $g\geq 2$ and suppose we have an short exact sequence of the form\begin{equation}\label{Extension}1\rightarrow S_g\rightarrow E\rightarrow A\rightarrow 1,\end{equation}
where $A$ is a finitely generated abelian group. Under what conditions is $E$ a K\"ahler group? 
\end{question}

We first make some preliminary observations.  If $G=\pi_1(X)$ and $H=\pi_1(Y)$ with $X$ and $Y$ K\"ahler manifolds, then $X\times Y$ is K\"ahler and hence so is $G\times H=\pi_1(X\times Y)$.  Therefore, since $S_g$ and $\Z^{2n}$ are each K\"ahler, $E=S_g\times \Z^{2n}$ gives one possible extension of $\Z^{2n}$ by $S_g$.  
Here we prove

\begin{thm}\label{Main} Let $g\geq2$, and suppose $E$ is a K\"ahler group which fits into an exact sequence of the form\[1\rightarrow S_g\rightarrow E\rightarrow A\rightarrow 1.\]Then there exists a finite index subgroup $E'\leq E$ such that $E'\cong S_g\times \Z^r$ with $r$ even.
\end{thm}

\begin{rmk} If $g=1$ the theorem is no longer true.  Indeed Campana \cite{Cam95} and Carlson--Toledo \cite{CaTo95} have shown that the Heisenberg group $\mathcal{H}_{2n+1}(\Z)$ is K\"ahler if and only if $n\geq 4$.  These are 2-step nilpotent, with center $\Z$ and abelianization $\Z^{2n}$.  Hence they are non-abelian but fit into exact sequences of the form:\[1\rightarrow \Z^2\rightarrow \mathcal{H}_{2n+1}(\Z)\rightarrow \Z^{2n-1}\rightarrow 1.\] 
\end{rmk}
Previously, versions of sequence (\ref{Extension}) were studied by Kapovich \cite{K13} and Py \cite{Py14} when $E$ is the fundamental group of an aspherical complex surface and $A=\Z^2$, in connection with coherence of complex hyperbolic lattices.  After we completed our paper, Pierre Py communicated an alternate proof of Theorem \ref{Main} using a different approach, in joint work with Gerardo Arizmendi.  

We also present a partial converse to Theorem \ref{Main}.  Any extension such as $E$ is determined by a homomorphism $\rho:A\rightarrow \Out(S_g)$, where $\Out(S_g)$ is the outer automorphism group of $S_g$.  By the Dehn--Nielsen--Behr theorem (see \cite{FaMa12}, Theorem 8.1), the mapping class group $\Mod(\Sigma_g)$ can be identified as an index 2 subgroup of $\Out(S_g)$.   After passing to a subgroup of finite index, we can obtain a representation $\rho:A\rightarrow \Mod(\Sigma_g)$.  

\begin{thm} \label{Realization}Suppose $A$ is a finitely generated abelian group with even rank, let $\rho:A\rightarrow \Mod(\Sigma_g)$ be a representation with finite image, and let $E$ be the corresponding extension. Then there exists a K\"ahler manifold $Z$ such that $\pi_1(Z)\cong E$.  Moreover, if $A$ is torsion-free then $Z$ can be taken to be aspherical.
\end{thm}

\noindent
\textbf{Acknowledgements.} The authors would like to thank the Mathematical Sciences Research Institute for its hospitality while this paper was being written.  The first author is extremely grateful to Mahan Mj for many illuminating discussions, to Misha Kapovich, for inspiring Theorem \ref{Realization}, and to Andy Putman, for suggesting the cover in Proposition \ref{PantsLift}. We are grateful to Pierre Py for explaining his and Gerardo Arizmendi's approach to proving Theorem \ref{Main}, and for pointing out reference \cite{BLM83} to us.  We would also like to thank Song Sun, Dieter Kotschick, and Brendan Hassett for several useful comments. The second author was partially supported through NSF grant DMS 1606387.  

\section{Restrictions on K\"ahler groups}
The proof of Theorem \ref{Main} will follow from a sequence of reductions, where in each case we restrict the range of possibilities for the extension $E$. Before going into the proof we need to discuss some known restrictions on K\"ahler groups. The first is that cocompact lattices in real hyperbolic space $\mathbb{H}^n$, $n\geq 3$ are not K\"ahler.  

\begin{thm}\label{hyperbolic}$($Carlson--Toledo \cite{CaTo89}$)$ Let $G$ be a K\"ahler group and $\Gamma\leq \Isom(\Hy^n)$ a cocompact lattice, for $n\geq 2$.  Then any homomorphism $\phi:G\rightarrow \G$ factors through a homomorphism $\phi':G\rightarrow S_g$ for some $g\geq 2.$ In particular, if $\Gamma\leq \Isom(\Hy^n)$ is a cocompact lattice for $n\geq 3$, then $\Gamma$ is not K\"ahler.
\end{thm}

%

The second restriction comes from the fact that compact K\"ahler manifolds are formal in the sense of rational homotopy theory.  For an introduction to rational homotopy theory, see \cite{GrMo81}.  Let $X$ be a smooth manifold, and denote by $\Omega^*(X)$ the de Rham complex of smooth $\C$-valued differential forms on $X$ (we could take any coefficients in any field of characteristic 0). Rational homotopy theory associates to $X$ a `minimal' differential graded algebra (d.g.a.) $(\mathcal{M}^*(X),d)$ together with a chain map $\rho_X:\mathcal{M}^*(X)\rightarrow \Omega^*(X)$ inducing an isomorphism on $H^*(X;\C)\cong H^*(\Omega^*(X))$.  $X$ is said to be \emph{formal} if $ (\mathcal{M}^*(X),d)\cong (H^*(X,\C)$,0) i.e. the chain complex whose underlying group is $H^*(X;\C)$ and with differential identically 0 in all degrees. More generally, we have 

\begin{definition} \label{nformal} $X$ is said to be \emph{$n$-formal} if there exists a minimal d.g.a. $(\mathcal{M}^*,d)$ together with a chain map $\rho:\mathcal{M}^*\rightarrow \Omega^*(X)$ such that 
\begin{enumerate}
\item $\rho$ induces an isomorphism on cohomology in degrees $\leq n$, and an injection of the subring generated by $\bigoplus_{i=1}^nH^i(\mathcal{M}^*)$.
\item The differential $d$ is 0 in degrees $\leq n$.
\end{enumerate}
\end{definition}
Any $(\mathcal{M}^*,d)$ satisfying $(1)$ above is called an \emph{$n$-minimal model} for $X$.  Since $H^1(X;\C)\cong H^1(\pi_1(X);\C))$ this implies that if $X$ is 1-formal, then $\pi_1(X)$ is 1-formal.  Essentially this means that any 1-minimal model for $X$ and hence $\pi_1(X)$ is determined by $H^1(X)$ and the cup-product map $H^1(X)\times H^1(X)\rightarrow H^2(X)$. We have

\begin{thm}\label{formal}$($Deligne--Griffiths--Morgan--Sullivan \cite{DGMS75}$)$ If $X$ is a compact K\"ahler manifold then $X$ is formal. In particular, $\pi_1(X)$ is 1-formal.
\end{thm}

For our purposes, it will not be important to calculate the 1-minimal model for $X$. However, we will require the following consequence of 1-formality due to Papadima--Suciu.  Suppose $\pi_1(X)=G$ is 1-formal and fits into an exact sequence \[1\rightarrow N\rightarrow G\rightarrow \Z\langle t \rangle \rightarrow 1,\] where $N$ is finitely generated. Let $X_\Z$ be the associated infinite cyclic cover of $X$.  We can identify $H_1(X_\Z;\C)$ with $H_1(N;\C)$ and regard it as an $\C[t,t^{-1}]$-module where $t$ acts by deck transformations.  Thus, $t:H_1(N;\C)\rightarrow H_1(N;\C)$ is a linear isomorphism, and we can consider its Jordan normal form.  Then we have

\begin{thm}\label{NoJordan} $($Papadima--Suciu \cite{PS10}$)$ Suppose $G$ as above is 1-formal.  If the action of the $t$ on $H_1(N;\C)$ has eigenvalue 1, the associated Jordan blocks are all of size 1.
\end{thm}

As an example, the 3-dimensional Heisenberg group $\mathcal{H}_3(\Z)$ can be thought of as the fundamental group of a torus bundle fibering over the circle.  More precisely, take $T^2\times [0,1]$ and identify $T^2\times\{0\}$ and $T^2\times \{1\}$ by the diffeomorphism represented by the matrix \[A=
\left(\begin{array}{cc} 1 &1\\
0&1
\end{array}\right).\] This gives an splitting $\mathcal{H}_3(\Z)\cong \Z^2\rtimes \Z$, where the $\Z$-factor acts by $A$ on $\Z^2=H_1(T^2)$.  Since $A$ is a Jordan block of size 2, $\mathcal{H}_3(\Z)$ is not 1-formal.  Note that the fact that the cokernel is $\Z$ is crucial.  In particular, all the higher-dimensional Heisenberg groups $\mathcal{H}_{2n+1}(\Z)$ for $n\geq 2$ are 1-formal by Carlson--Toledo \cite{CaTo95} but like $\mathcal{H}_3(\Z)$, they factor as semi-direct products $\mathcal{H}_{2n+1}(\Z)=\Z^{n+1}\rtimes \Z^n$ where each generator of $\Z^n$ acts by a unipotent matrix with a single Jordan block of size 2.

\section{Abelian subgroups of $\Mod(\Sigma_g)$ and covers}
First we recall some basic facts about mapping class groups.   Let $\Sigma=\Sigma^b_{g,n}$ be the orientable surface of genus $g$, with $n$ marked points and $b$ boundary components. The \emph{mapping class group of $\Sigma$}, denoted $\Mod(\Sigma)$, is the group of orientation preserving diffeomorphisms of $\Sigma$, up to homotopy.  Diffeomorphisms and homotopies are required to fix $\partial \Sigma$ pointwise, and to fix marked points setwise. 
The Nielsen-Thurston classification of surface diffeomorphisms states that every mapping class $\phi$ falls into one of three categories: finite order, reducible, and pseudo-Anosov (see \cite{FaMa12}, Theorem 13.2). If $\phi$ has finite order, then $\phi$ has a representative which is a finite order diffeomorphism of $\Sigma$.  If $\phi$ is \emph{reducible}, then $\phi$ has a representative which fixes some 1-submanifold setwise. If $\phi$ is \emph{pseudo-Anosov}, then $\phi$ does not preserve any conjugacy class in $S_g$.  Finite order and reducible are not mutually exclusive, but both are disjoint from pseudo-Anosov. A subgroup $H\leq\Mod(\Sigma)$ is called reducible if every element $H$ fixes the same 1-submanifold of $\Sigma$.   

Now let $\Sigma=\Sigma_g$ be the closed, orientable surface of genus $g$.  The two results in the previous section are the main tools used in the proof of the main theorem, along with some understanding of abelian subgroups of $\Mod(\Sigma)$.  Suppose \[1\rightarrow S_g\rightarrow E\rightarrow Q\rightarrow 1\] be any extension of $Q$ by $S_g$.  It may not be split, but taking a set theoretic section $s:Q\rightarrow E$, the conjugation action of $E$ on $S_g$ induces a well-defined homomorphism $\rho:Q\rightarrow \Out(S_g)=\Mod^{\pm}(\Sigma)$, where the latter equality follows from the Dehn--Nielsen--Baer theorem.  Passing to a subgroup of index $2$ if necessary, we can thus assume that the image of $\rho$ lies in $\Mod(\Sigma)$.  In our case, $Q=\Z^r$ is abelian, so we just need to analyze free abelian subgroups of $\Mod(\Sigma)$. The following proposition follows immediately from work of Birman--Lubotzky--McCarthy \cite{BLM83}, but for reference we sketch the proof here.


\begin{proposition}\label{Split}Let $A\leq \Mod(\Sigma)$ be a free abelian subgroup.  Then \begin{enumerate} 
\item $A$ has rank at most $3g-3$.  
\item Let $E$ be the extension corresponding to $A$.  There is a finite index subgroup $E'\leq E$ which is split, i.e. can be written $E'=S_g\rtimes \Z^r$.   
\end{enumerate}
\end{proposition} 
\begin{proof}   If $\Z^r\leq \Mod(\Sigma)$ is not reducible, then the image of $\rho$ lies in the centralizer of a pseudo-Anosov diffeomorphism, which is virtually cyclic.  It follows that a finite index subgroup of $E$ splits as a product $(S_g\rtimes \Z)\times \Z^{r-1}$ where the $\Z$-factor is represented by a pseudo-Anosov diffeomorphism of $\Sigma$.  If $\Z^r$ is reducible, then after passing to a finite index subgroup, $\Z^r$ fixes a collection of disjoint simple closed curves $\mathcal{C}$, and each complementary region of $\Sigma\setminus \mathcal{C}$, up to homotopy.  This gives the bound $r\leq 3g-3$, proving (1). In the reducible case, it is easy to see that $\Z^r$ can actually be realized as a group of diffeomorphisms of the surface which act as the identity on a neighborhood of some chosen base-point $*$.  Identifying $\Mod(\Sigma,*)$ with $\Aut(S_g)$, we see that $E$ splits after passing to a finite index subgroup, proving (2). 
\end{proof}

From the previous proposition, without loss of generality we may assume that $E=S_g\rtimes \Z^r$.  Moreover, from the proof we have that either $\rho(\Z^r)$ lies in the centralizer of pseudo-Anosov diffeomorphism, or $\rho(\Z^r)$ is realized as group of surface diffeomorphisms which preserve a small tubular neighborhood of an embedded 1-submanifold $\mathcal{C}\subset\Sigma$.  Call this neighborhood $N(\mathcal{C})$.  Then every element of $\Z^r$ preserves each connected component of $N(\mathcal{C})$ and of $\Sigma\setminus N(\mathcal{C})$.  We assume that only one curve of $\mathcal{C}$ lies in each connected component of $N(\mathcal{C})$. The submanifold $\mathcal{C}$ can be chosen maximally so that after passing to a finite index subgroup, on each component of $\Sigma\setminus N(\mathcal{C})$, every element of $\Z^r$ acts either as the identity or as a pseudo-Anosov diffeomorphism (see \cite{FaMa12}, Corollary 13.3). 

We now describe a way to pass to finite index subgroups of $E$.  Let $\mathcal{C}$ denote the set of simple closed curves fixed by $\Z^r$.  Every component of $\Sigma\setminus \mathcal{C}$  has Euler characteristic negative.  We have the following proposition about passing to covers of $\Sigma$.

\begin{proposition} \label{lift}Let $E$ be as above.  For any cover $\widetilde{\Sigma}$ of $\Sigma$, there exists a finite index subgroup $\widetilde{E}\leq E$ with $\widetilde{E}\cong \pi_1(\widetilde{\Sigma})\rtimes \Z^r$.  
\end{proposition}
\begin{proof} Fix a cover $p:\widetilde{\Sigma}\rightarrow \Sigma$ of degree $d$.   Each curve $C\in \mathcal{C}$ has pre-image some union of curves $\tilde{C}_1,\cdots, \tilde{C}_k$, and each subsurface $W\subset \Sigma$ has pre-image some union $\tilde{W}_1,\ldots, \tilde{W}_l$ of covers of $W$.  Let $e_1,\ldots, e_r$ be generators for $\Z^r$.  We claim that there exists $D>0$ such that the subgroup generated by $D$th powers $\langle e_1^D,\ldots e_r^D\rangle\leq \Z^r$ lifts to $\widetilde{\Sigma}$.  We may assume that if $e_i$ and $e_j$ act as a Dehn twist along some curve $C\in\mathcal{C}$, then $e_i$ and $e_j$ both act as powers of the same fixed diffeomorphism. 

To prove the claim, we lift each $e_i$ separately. There are two cases to consider.  \\

\noindent
\textbf{Case 1:} Suppose $e_i$ acts as a Dehn twist along some curve $C$.  The pre-image $p^{-1}(C)=\tilde{C}_1\coprod\cdots\coprod \tilde{C}_k$ is a disjoint union of simple closed curves equipped with covering maps $p_j=p|_{\tilde{C}_j}:\tilde{C}_j\rightarrow C$.  If the $p_j$ has degree $q_j$ then $\sum_jq_j=d$, and $q_j|d$.  By the observation above, $e_i$ is a $k$th power of some fixed Dehn twist diffeomorphism $T_C$ supported on a neighborhood $N$ of $C$.  Then $e_i^d$ lifts to a diffeomorphism which acts as $T_{\tilde{C}_j}^{(k\cdot d/q_j)}$ on a neighborhood $\tilde{N}_j$ of $\tilde{C}_j$. \\

\noindent
\textbf{Case 2:} Suppose some $e_i$ acts as a pseudo-Anosov diffeomorphism on some complementary subsurface $W$.  Then for any other $e_j, j\neq i$, $e_j$ either acts as the identity on $W$ or $e_i$ and $e_j$ are both powers of a common pseudo-Anosov $f_W$, which restricts to the identity on $\partial W$. We therefore lift $f_W$.  As in the case of a curve, $p^{-1}(W)=\tilde{W}_1\coprod\cdots\coprod \tilde{W}_l$ is a disjoint union of covers of $W$, with $W_j$ have degree $t_j$, say.  The image of $\pi_1(\tilde{W}_j)$ in $\pi_1(W)$ is a subgroup of index $t_j$.  By choosing a sufficiently high power $D_i$ of $f_W$ we may insure that $f_W^{D_i}$ acts as the identity permutation on all subgroups of $\pi_1(W)$ of index at most $\max\{t_j\}$. Then $f_W^{D_i}$, and hence $e_i^{D_i}$, lifts to each $\tilde{W}_j$. \\

Now set $D=\max\{D_1,\ldots, D_r,d\}$.  It follows that the subgroup $\langle e_1^D,\ldots, e_r^D\rangle$ preserves the subgroup $\pi_1(\widetilde{\Sigma})\leq \pi_1(\Sigma)$.  Thus, there is a subgroup $\widetilde{E}\leq E$ of index at most $d\cdot D^r$ such that $\widetilde{E}\cong \pi_1(\widetilde{\Sigma})\rtimes \Z^r$, as desired.
\end{proof}

\begin{rmk}\label{StillPseudo}  Since we considered a cover in Proposition \ref{lift}, the lifts are type preserving in the following sense.  If some element $v\in\Z^r\leq E$ acts as a pseudo-Anosov on some subsurface $W\subset \Sigma$, and $v$ lifts $\widetilde{\Sigma}$, then $v$ acts as a pseudo-Anosov on every component of the pre-image of $W$ in $\widetilde{\Sigma}$.  This follows from the fact that pseudo-Anosov diffeomorphisms are characterized by the property that they preserve a pair of transverse measured foliations on the surface, so the same is true in any lift to a cover.
\end{rmk}

We will also need a recent result of Hadari about the action of a mapping class on the homology of a cover.  A surface $\Sigma$ is said to have \emph{finite type} if it has finitely many marked points and boundary components.

\begin{thm}\label{InfiniteOrder} $($Hadari \cite{Had15}$)$ Let $\Sigma$ be an oriented surface of finite type with free fundamental group.  If $\phi\in \Mod(\Sigma)$ is a mapping class with infinite order then there exists a finite cover $\widetilde{\Sigma}\rightarrow \Sigma$ and a lift $\widetilde{\phi}$ of $\phi$ such that the action of $\widetilde{\phi}$ on $H_1(\widetilde{\Sigma})$ has infinite order.  
\end{thm}

 In fact, Hadari shows that $\widetilde{\Sigma}$ can be taken to be a solvable cover, but all that will be important for us is that $\widetilde{\phi}$ has infinite order.

\begin{figure}[!ht]
\centering
\vspace*{-.25in}
\hspace*{.25in}
\includegraphics[width=6in]{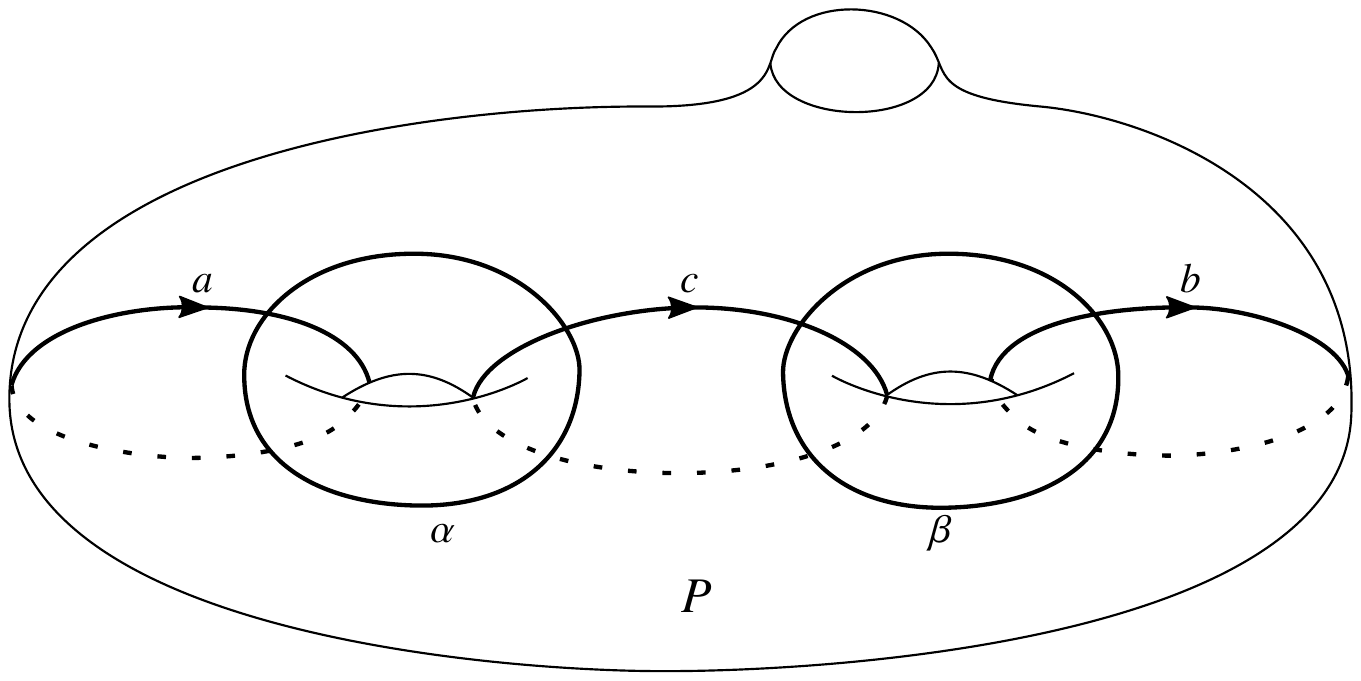}
\vspace{-6in}
\caption{A pair of pants $P\subset \Sigma$ with boundary curves $a$, $b$ and $c$.  With the indicated orientations $[a]+[b]+[c]=0\in H_1(\Sigma)$. The curves $\alpha$ and $\beta$ are Poincar\'{e} dual to $a$ and $b$, respectively.}
\label{Fig:Pants}

\end{figure}

Finally, we need a result about Dehn twists along pants curves and passing to covers.  Let $\Sigma$ be a closed, oriented surface and $P\subset \Sigma$ a pair of pants, with boundary curves $a,b,c$ in which every every pants curve is non-separating.  See figure \ref{Fig:Pants} for a schematic.  Then we have $[a]+[b]+[c]=0$ in $H_1(\Sigma;\C)$.  Denote by $T_a$, $T_b$ and $T_c$ are the right-handed Dehn twists about $a$, $b$, and $c$, respectively.  The action of $T_a$, $T_b$ and $T_c$ on $H_1(\Sigma;\C)$, makes it into an $\C[t_a^{\pm},t_b^{\pm},t_c^\pm]$-module.  Observe that the action of $t_c$ on the quotient module $H_1(\Sigma;\C)/\langle t_a,t_b\rangle$ is trivial, and similarly for any pair of $T_a,$ $T_b,$ and $T_c$.  The next proposition shows that after passing to a cover, this need not be the case.

\begin{proposition} \label{PantsLift}Let $P\subset \Sigma$ be a pair of pants as above, and suppose $T_a^{n_a}$, $T_b^{n_b}$ and $T_c^{n_c}$ are powers of Dehn twists about $a$, $b$, and $c$, respectively, for some $n_a,n_b,n_c\in \Z$. Then there exists a cover $\widetilde{\Sigma}\rightarrow\Sigma$ and, after passing to powers, lifts $\widetilde{T}_a$, $\widetilde{T}_b,$ and $\widetilde{T}_c$ such that $\widetilde{T}_c$ acts as a non-trivial Dehn twist on on $H_1(\widetilde{\Sigma};\C)/\langle \widetilde{t}_a,\widetilde{t}_b\rangle$.
\end{proposition}

\begin{proof} We assume $a$, $b$ and $c$ are as in figure \ref{Fig:Pants}.  Find curves $\alpha$ and $\beta$ Poincar\'{e} dual to $a$ and $b$ respectively.  Then $\langle [a],[\alpha], [b], [\beta]\rangle$ generate a rank 4 symplectic subspace $V$ of $H_1(\Sigma)$.  Projection onto $V$ followed by reduction modulo 2 defines a homomorphisms $\varphi:\pi_1(\Sigma)\twoheadrightarrow (\Z/2\Z)^4$. Let $p:\widetilde{\Sigma}\rightarrow \Sigma$ be the 16-fold cover corresponding to $\varphi$. Each of $a, \alpha, b, \beta, c$ lift to 8 simple closed curves, where the restriction of $p$ to each curves is a 2-fold cover of the corresponding curve in $\Sigma$.

One can visualize the cover $\widetilde{\Sigma}$ as follows.  Let $F_1$ be the torus with one boundary component formed by taking a tubular neighborhood of $a\cup \alpha$, and let $F_2$ be the corresponding torus formed by $b\cup \beta$.  If $d_i$ is the boundary curve of $F_i$, then $c$ meets $d_i$ in two points.  Restricted to $F_i$, $p$ is the characteristic cover corresponding to the quotient $\Z^2\twoheadrightarrow(\Z/2\Z)^2$.  To see what this looks like, observe that the universal abelian cover of $F_i$ is just $\R^2$ with the interior of a small disk removed at each integer lattice point.  We then quotient by the subgroup $2\cdot\Z^2\leq \Z^2$ to get the characteristic cover.  Thus, the characteristic cover can be obtained by taking a square with 4 disks in the interior removed, and  identifying opposite sides by a translation. 

Denote the characteristic cover of $F_i$ by $\widetilde{F}_i$.  The preimage of $F_i$ is 4 disjoint copies of $\widetilde{F}_i$.  Each of the boundary curves in the copies of $\widetilde{F}_i$ is a different lift of $d_i$. Since $d_i$ is in the kernel of $\varphi$, there is exactly one lift for every element of $(\Z/2\Z)^4$. Observe that if $W=\Sigma\setminus (F_1\cup F_2)$, then $\varphi|_{\pi_1(W)}=1$, and $W$ has 16 disjoint lifts to $\widetilde{\Sigma}$, each isomorphic to $W$ itself.  To form the cover, $\widetilde{\Sigma}$, we join the lifts of $d_1$ to the corresponding lifts of $d_2$ along a lift of $W$. This will be $\widetilde{\Sigma}$.  

As $\varphi([c])\neq0$, the preimage of $p^{-1}(c)$ is a union of 4 nonseparating simple closed curves each of which double covers $c$ via $p$.  We claim that there is a nonseparating simple closed curve on $\widetilde{\Sigma}$ on which meets some lift of $c$ exactly once, and which doesn't meet $p^{-1}(a)$ or $p^{-1}(b)$. Let $\widetilde{c}$ be one of the components of the preimage of $c$.  There are exactly four preimages of $d_1$ which meet $c_1$ exactly once, and each is nonseparating. Choose one such curve and call it $\widetilde{d}_1$. Since $d_1$ doesn't meet a or $c$, we must have that $\widetilde{d}_1$ doesn't meet $p^{-1}(a)$ or $p^{-1}(b)$.

To prove the proposition, we now choose a symplectic basis $\mathcal{B}$ for $H_1(\widetilde{\Sigma})$ which extends $\{[\widetilde{c}],[\widetilde{d}_1]\}$. By Proposition \ref{lift}, we can lift the action of some power of $T_a^{n_a}$, $T_b^{n_b}$ and $T_c^{n_c}$ to $\widetilde{\Sigma}$ to $\widetilde{T}_a$, $\widetilde{T}_b$, and $\widetilde{T}_c$, respectively. With respect to $\mathcal{B}$, both $\widetilde{T}_a$ and $\widetilde{T}_b$ act trivially on the symplectic subspace $U$ spanned by $\{[\widetilde{c}],[\widetilde{d}_1]\}$. Therefore, $U\hookrightarrow H_1(\widetilde{\Sigma})$ and the action of $\widetilde{T}_c$ on the image of $U$ in $H_1(\widetilde{\Sigma})$ is an infinite order Dehn twist of $[\widetilde{d}_1]$ about $[\widetilde{c}]$.  \\
\end{proof}
\section{Proof of the main theorem}
In this section we prove the main theorem.  The strategy is to assume we have a K\"ahler extension $E$, and then construct finite index subgroups and quotients of $E$ with contradictory properties unless $E$ is virtually a product. By Proposition \ref{Split}, we assume that $E=S_g\rtimes \Z^r$, where the action of $\Z^r$ is given by a faithful representation into $\Mod(\Sigma_g)$.  Write $ \{e_1,\ldots, e_r\}$ for the standard basis of $\Z^r$, and $\mathcal{C}$ the disjoint collection of simple closed curves preserved pointwise by $\Z^r$ on $\Sigma=\Sigma_g$.  Observe that each component of $\Sigma\setminus \mathcal{C}$ is a subsurface with negative Euler characteristic, and at worst $\mathcal{C}$ provides a pants decomposition for $\Sigma$.  Armed with this observation, we are now able to prove Theorem \ref{Main}:

\begin{proof} Assume that $\pi_1(X)=E=S_g\rtimes \Z^r$, where $X$ is a compact K\"ahler manifold.  There are two cases to consider depending on whether $\mathcal{C}$ is empty.  If $\mathcal{C}$ is empty, then after changing the basis for $\Z^r$, we have that  $e_1$ acts as a pseudo-Anosov on $\Sigma$ and $E\cong (S_g\rtimes \Z)\times \Z^{r-1}$.  Since $g\geq 2$, a well-known theorem of Thurston (see \cite{FaMa12}, Theorem 13.4) implies that $S_g\rtimes \Z$ is the fundamental group of a closed hyperbolic 3-manifold.  Hence we may project $\pi:E\twoheadrightarrow S_g\rtimes \Z$.  Theorem \ref{hyperbolic} implies that $\pi$ must factor through a surface group, but $\pi|_{S_g\rtimes \Z}$ is an isomorphism, so $S_g\rtimes \Z$ must be a subgroup of a surface group.  Since the latter has cohomological dimension 2, while $S_g\rtimes \Z$ has cohomological dimension 3, this is impossible.  

Thus we may assume $\mathcal{C}$ is non-empty.  Consider the components of $\Sigma\setminus N(\mathcal{C})=W_1\coprod\cdots\coprod W_k$.  As observed, each $W_i$ is a surface with non-empty boundary and Euler characteristic $\leq -1$, where $\chi(W_i)=-1$ only if $W_i$ is a pair of pants.  Given a subsurface $W_i$ with boundary $C_1,\ldots,C_n$, define the $\overl{W_i}$ to be the surface obtained from $W_i$ by capping each $C_i$ with a disk.  Let $g_i$ be the genus of $\overl{W_i}$ and consider the complement $W_i'=\Sigma\setminus W_i$.  There is a quotient of $S_g$ obtained by surgering disks along curves in $W'$ so that it kills the fundamental group of $W'$. This provides a surjection from $S_g$ onto $S_{g_i}=\pi_1(\overl{W_i})$.  Let $K_i$ be the kernel of this homomorphism.  Since the action of $\Z^r$ preserves the boundary of $W_i$ pointwise,  $K_i\leq S_g$ is normal in $E$, and hence we obtain a quotient map $\kappa:E\rightarrow Q$ and a short exact sequence:
\[1\rightarrow K_i\rightarrow E\xrightarrow{\kappa} Q\rightarrow 1.\]
Observe that $K_i\cap\Z^r=\{1\}$, hence $\Z^r$ passes isomorphically to $Q$.  Let $V_i\leq \Z^r$ be the subgroup which acts trivially on $W_i$.  Note that because there can be at most one pseudo-Anosov acting on $W_i$, $V_i$ is a direct summand of $\Z^r$ isomorphic to either $\Z^r$ or $\Z^{r-1}$. Hence, we see that $Q$ splits as a product $Q=\left(S_{g_i}\rtimes(\Z^r/V)\right)\times V$.  

Even if some element of $\Z^r$ acts on $W_i$ by a pseudo-Anosov, after capping, the action of $\Z^r/V$ on $\overl{W_i}$ may no longer be pseudo-Anosov.  For example, if $W_i$ has a single boundary component, a theorem of Kra \cite{Kra81} states one can obtain a pseudo-Anosov by point-pushing along a curve which fills the surface.  Rel boundary this is pseudo-Anosov, but after capping, it follows from the Birman exact sequence \cite{B69} that the action is trivial.  However, we do get the same action on homology.  

\begin{lemma}\label{sameHomology}Let $\phi$ be a diffeomorphism of a compact, orientable surface $W$ fixing $\partial W$ pointwise, and let $\overl{\phi}$ be the induced map on $\overl{W}$.  Then there is a commutative diagram 
\[\xymatrix{H_1(W)\ar[d] \ar[r]^{\phi}& H_1(W)\ar[d]\\
H_1(\overl{W}) \ar[r]^{\overl{\phi}}&H_1(\overl{W})}\]
\end{lemma}
\begin{proof} Let $B\subset H_1(W)$ be the subgroup generated by the boundary curves.  Since $\phi$ acts trivially on the boundary, it preserves the summand $B$ pointwise. But $B$ is exactly the kernel of the map $H_1(W)\rightarrow H_1(\overl{W})$.
\end{proof}

From the lemma we see that on the level of homology, the action of $\Z^r/V$ on $\overl{W_i}$ is the same as on $W_i$.  We now proceed as follows.  First suppose that $\Z^r$ acts as a pseudo-Anosov on some subsurface $W_1$. After changing the basis for $\Z^r$, we may suppose that $e_1$ acts as a pseudo-Anosov on the interior of $W_1$, and every other basis element acts as the identity.  
By passing to a cover, we can guarantee that by Proposition \ref{lift}, the genus of $\overl{W_1}$ is at least 2, and by Theorem \ref{InfiniteOrder}, the action of $e_1$ on $H_1(W_1)$ has infinite order.  Now we cap $W_1$ to get $\overl{W_1}$, together with an induced homeomorphism $\overl{e_1}$.  If $\overl{e_1}$ is still pseudo-Anosov on $\overl{W_1}$, then the discussion above implies that $E\twoheadrightarrow S_{g_1}\rtimes(\Z^r/V)=S_{g_1}\rtimes\Z$. The latter is the fundamental group of a closed hyperbolic 3-manifold, hence by Theorem \ref{hyperbolic}, we have a factorization \begin{equation}\label{Factorization}\begin{aligned} \xymatrix{
E\ar[dr]_{\psi}\ar[rr]^\pi &&  S_{g_1}\rtimes \Z\\
 & S_h\ar[ur]}
\end{aligned}
\end{equation} for some $h\geq 2$. We claim that no such factorization exists. 

\begin{lemma}\label{NoFactorization} There does not exist $\psi$ making diagram $(\ref{Factorization})$ commute.  
\end{lemma}
\begin{proof} First observe that $\langle e_1\rangle$ injects into $S_h$.  Since $S_h$ does not contain any $\Z^2$ subgroups, any element of $E$ that commutes with $e_1$ must lie in $\ker\psi$. In particular, $V\subset \ker \psi$, as well as any element of $S_g$ represented by a curve which is freely homotopic to some curve disjoint from $W_1$.  It follows that $\ker\pi\subseteq \ker\psi$, hence $\ker\pi=\ker\psi$. Hence, in order for $\psi$ to exist we must have that $S_{g_1}\rtimes \Z\leq S_h$, which is impossible as the latter has cohomological dimension 2, while the former has cohomological dimension 3.  
\end{proof}

Therefore, if $\overl{e_1}$ is still pseudo-Anosov on $\overl{W_1}$, then Theorem \ref{hyperbolic} implies we have a factorization as in diagram (\ref{Factorization}), which is impossible by Lemma \ref{NoFactorization}. Otherwise, the action of $\overl{e_1}$ on $\overl{W_1}$ has infinite order in homology but is reducible.  By Lemma \ref{sameHomology}, the action of $e_1$ on $H_1(W_1)$ has infinite order but reducible characteristic polynomial.  If the characteristic polynomial has some irreducible factor which is not linear, then the action of a power of $e_1$ preserves some connected subsurface $W_1'\subset W_1$, and acts as a pseudo-Anosov on that subsurface, by the Casson--Bleiler homological criterion \cite{CaBl88}.  Now we pass to a finite cover so that the genus of $W_1'$ is at least 2.  Capping $W_1$ then subsequently $W_1'$, we know that the action of $\overl{e_1}$ on $\overl{W_1}'$ is pseudo-Anosov. In this case we also obtain a quotient of $E$ which is the fundamental group of a closed hyperbolic 3-manifold, which is a contradiction by the same argument as above. 

If the action of $\overl{e_1}$ is unipotent, then on the level of homology, $\overl{e_1}$, and hence $e_1$ looks like a multitwist which acts non-trivially on homology.  In this case, we consider the exact sequence \[1\rightarrow S_g\rtimes V\rightarrow E\rightarrow \Z\langle e_1\rangle\rightarrow 1.\]
Since $V$ acts trivially on $H_1(W_1)$, we know that $H_1(W_1)\hookrightarrow H_1(S_g\rtimes V)$ as a direct summand. Denote by $e_1^*$ the action of $e_1$ on $H_1(S_g\rtimes V)$.  The action of $e_1^*$ preserves $H_1(W_1)$ and if we apply Theorem \ref{NoJordan}, we see that the action of $e_1$ on $H_1(W_1;\C)$ cannot have any Jordan blocks of size $>1$ associated to the eigenvalue $1$.  On the other hand, $e_1^*$ is unipotent, hence every eigenvalue is 1.  It follows that $e_1^*-I=0$ on $H_1(W_1;\C)$, and contradicting the assumption that $e_1^*$ has infinite order.  

At this point, we may assume no element of $\Z^r$ acts as a pseudo-Anosov.  In this case, $\Z^r$ acts purely by multitwists and it is not hard to build a cover where $\Z^r$ acts non-trivially on $H_1$.  Alternatively, identifying $\Aut(S_g)$ with $\Mod(\Sigma_{g,1})$ we may apply Hadari's theorem again to get an infinite order action by multitwists on homology. Choose a symplectic basis $\langle a_1,\ldots a_g, b_1, \ldots, b_g\rangle$ for $H_1(\Sigma)$ which contains all of the non-separating curves that $\Z^r$ twists along. With respect to this basis, we see that the image of the representation of $\Z^r$ in $\Sp_{2g}(\Z)$ is block upper triangular of the form \[\left(\begin{array}{cc} I_g & A\\
0 & I_g\end{array}\right).\]
where $A^t=A$ is symmetric. For a basis $e_1,\ldots ,e_r$ for $\Z^r$, let $A_1,\ldots A_r$ be the corresponding upper right blocks under the representation.  Using the Euclidean algorithm, we can change our basis for $\Z^r$ so that the following two conditions hold.
\begin{enumerate}
\item Either for some $i$, $(A_1)_{ii}\neq0$, or for some $i\neq j$,  $(A_1)_{ij}=(A_1)_{ji}\neq0$.
\item For all $2\leq k\leq r$, correspondingly we have $(A_k)_{ii}=0$ or $(A_k)_{ij}=(A_k)_{ji}=0$.  
\end{enumerate}
Without loss of generality, assume that either $(A_1)_{11}\neq 0$ or $(A_1)_{12}\neq 0$.  We will eliminate the case when $(A_1)_{11}\neq 0$; the other is similar. For each $j>1$, consider the action of $\Z^r$ on the symplectic subspace $L_{1j}=\langle a_1,b_1, a_j,b_j\rangle$.  Any free abelian subgroup acting on this subspace faithfully has rank $\leq 3$.  Therefore, after changing our basis once again, there are at most 3 basis vectors including $e_1$ which act non-trivially on $L_{1j}$. If $e_1$ is the only basis vector that acts non-trivially on $L_{1j}$, then if we consider the exact sequence
\[1\rightarrow S_g\rtimes V\rightarrow E\rightarrow \Z\langle e_1\rangle\rightarrow 1,\]
we know that $L_{1j}\hookrightarrow H_1(S_g\rtimes V;\C)$.  Hence $H_1(S_g\rtimes V;\C)$ will have a subspace of dimension at least 2 on which $e_1$ acts as a non-trivial Jordan block of size at least 2, and eigenvalue 1. In this case we may apply Theorem \ref{NoJordan} to conclude that $E$ cannot be K\"ahler. Similarly, if for some $j$, there are exactly two basis vectors which act non-trivially on $L_{1j}$, say $e_1$ and $e_2$, then $(L_{1j}/\langle e_2)\rangle\otimes \C$ will be three dimensional, hence the image of $L_{1j}$ in $H_1(S_g\rtimes V;\C)$ will still be a subspace on which $e_1$ acts as a non-trivial Jordan block of size at least 2 and eigenvalue 1. 

In the final case, for some $j$ we have exactly 3 non-trivial basis vectors, $e_1$, $e_2$ and $e_3$ say, acting on $L_{1j}$.  After changing the basis for $K=\langle e_1,e_2, e_3\rangle$, we can arrange that $(A_1)_{11}\neq0$, $(A_2)_{1j}\neq0$ and $(A_3)_{jj}\neq 0$, while all other entries among these three generators are 0.  Now we are in the situation of Proposition \ref{PantsLift}, since on $V_{1j}$, the generator $e_1$ acts as a power of a Dehn twist about $[a_1]$, $e_2$ acts as a power of a Dehn twist along $[a_1]+[a_j]$ and $e_3$ acts as a power of a Dehn twist along $[a_j]$. Applying Proposition \ref{PantsLift}, we pass to a finite index subgroup $\widetilde{E}$ of $E$, and correspondingly, a finite index subgroup $\widetilde{K}$ of $K$. Write $\widetilde{E}=S_{\tilde{g}}\rtimes\Z^r$, and let $\widetilde{V}$ be $\langle e_1, e_3,\ldots, e_k\rangle \cap \widetilde{E}$.   
Let $\widetilde{e_2}$ be the smallest multiple of $e_2$ contained in $\widetilde{K}$.  Then Proposition \ref{PantsLift} implies that if we consider the exact sequence \[1\rightarrow S_{\tilde{g}}\rtimes\widetilde{V}\rightarrow \widetilde{E}\rightarrow \Z\langle\widetilde{e_2}\rangle\rightarrow 1,\]
the action of $\widetilde{e_2}$ on $H_1(S_{\tilde{g}}\rtimes\widetilde{V};\C)$ will have Jordan blocks of size at least 2 and with eigenvalue 1.  This contradicts Theorem \ref{NoJordan}, eliminating this case as well.

Therefore, we must have that the image of $\Z^r$ in $\Mod(\Sigma)$ is finite. Hence, after passing to a finite index subgroup $\widetilde{E}$, we have that $\Z^r$ acts trivially on $\pi_1(\Sigma)$, and $\widetilde{E}=S_g\times \Z^r$ is a product.  The fact that $E$ was assumed to be K\"ahler implies that the first Betti number of $\widetilde{E}$ is even.  Thus, $2g+r$ is even, and so $r$ is also even, as desired.
\end{proof}

\section{Nielsen Realization and Virtual Products}
In this section we prove Theorem \ref{Realization}, which provides a partial converse to Theorem \ref{Main}. First we recall some background on the Nielsen realization problem and Kerckhoff's subsequent solution.  Recall that by the Dehn--Nielsen--Baer theorem, the mapping class group $\Mod(\Sigma)$ can be identified with an index 2 subgroup of $\Out(S_g)$. Geometrically, $\Mod(\Sigma)$ is the group of orientation-preserving outer automorphisms, i.e. those which induce the identity on $H_2(S_g)$. Since $\Sigma$ is a $K(S_g,1)$, we can further identify $\Out(S_g)$ with $\HE(\Sigma)$, the set of self-homotopy equivalences of $\Sigma$.  A homeomorphism $f:\Sigma\rightarrow \Sigma$ is said to \emph{realize} $\phi\in \Out(S_g)$ if the homotopy class $[f]\in \HE(\Sigma)$ is the same as $\phi$ under the above identification.  

Given a finite subgroup $F\leq \Mod(\Sigma)$, Nielsen asked whether there always exists a Riemann surface $C$ of genus $g$ on which $F$ is realized as a group of isometries. This question became known as the \emph{Nielsen realization problem}.  For $g=1$, any finite subgroup of $\Mod_1\cong \SL_2(\Z)$ is isomorphic to a subgroup of $\Z/6\Z$ or $\Z/4\Z$.  Using the fundamental domain for the action of $\SL_2(\Z)$ on the upper half-plane $\Hy^2$, it is easy to see that indeed $\Z/4\Z$ and $\Z/6\Z$ are realized as the full isometry group of the square and hexagonal torus, respectively.  In a groundbreaking paper, Kerckhoff solved the Nielsen realization problem for $g\geq 2$:

\begin{thm}\label{Kerckhoff}$($Kerckhoff \cite{Ker83}$)$ Given any finite subgroup $F\leq \Mod(\Sigma)$, where $g\geq 2$, there exists a hyperbolic surface $C$ such that $F$ is realized as a subgroup of $\Isom^+(C)$.
\end{thm}

We also require a result of Serre about finite K\"ahler groups. 
\begin{thm} \label{Serre}$($Serre \cite{Ser58}$)$ Every finite group $F$ is the fundamental group of some compact K\"ahler manifold.  
\end{thm}

Putting these two results together, we are able to prove Theorem \ref{Realization}.
\begin{proof}Let $A$ be any finitely generated abelian group of even rank, $\rho:A\rightarrow \Mod(\Sigma_g)$ be any homomorphism with finite image and $E$ the associated extension.  By the fundamental theorem for finitely generated abelian groups, we can decompose $A$ as $\Z^{2s}\times T$, where $T$ is the torsion subgroup of $A$.  By Theorem \ref{Serre}, there exists a compact K\"ahler manifold $X$ with $\pi_1(X)\cong A$.  Let $\widetilde{X}$ be the universal cover of $X$, and let $\omega_{\widetilde{X}}$ be the pullback of the K\"ahler form on $X$ to $\widetilde{X}$. The standard embedding of $\Z^{2s}\hookrightarrow \C^s$ gives an action of $\Z^{2s}$ on $\C^s$ by translation, which preserves the standard K\"ahler form $\omega_0$ on $\C^s$.Thus, $A$ acts freely, properly discontinuously, cocompactly on $\widetilde{Y}=\C^s\times \widetilde{X}$.  

By Theorem \ref{Kerckhoff}, there exists hyperbolic surface $C$ of genus $g$ on which $\rho(A)$ is realized as a group of isometries.  The volume form $\omega_C$ is a K\"ahler form on $C$ which is clearly invariant under the action of $\Isom^+(C)$. By an abuse of notation, we will also denote the representation $A\rightarrow \Isom^+(C)$ by $\rho$.  

To prove the theorem, we appeal to a K\"ahler variant of the Borel construction. Define $\widetilde{Z}=\widetilde{Y}\times C$. We then consider the diagonal action of $A$ on $\widetilde{Z}$, where $A$ acts on $\widetilde{Y}$ by deck transformations and on $C$ via $\rho$.  Since $A$ acts freely properly discontinuously cocompactly on $\widetilde{Y}$, the same is true of the action on $\widetilde{Z}$.  The quotient manifold $Z$ is clearly a complex manifold with $\pi_1(Z)\cong E$, and we claim it is in fact K\"ahler.  Indeed, consider the K\"ahler form $\omega_{\widetilde{Z}}=\omega_0+\omega_{\widetilde{X}} +\omega_C$ on $\widetilde{Z}$.  By construction, for every $a\in A$, we have \begin{align*} a^*(\omega_{\widetilde{Z}})&=a^*(\omega_0+\omega_{\widetilde{X}} +\omega_C)\\
&=a^*(\omega_0)+a^*(\omega_{\widetilde{X}}) +\rho(a)^*(\omega_C)\\
&=\omega_0+\omega_{\widetilde{X}} +\omega_C\\
&=\omega_{\widetilde{Z}}.
\end{align*}
Hence, $\omega_{\widetilde{Z}}$ descends to a K\"ahler form $\omega_Z$ on $Z$, which proves that $Z$ is K\"ahler.  For the final statement of the theorem, observe that in the case where $A$ is torsion free, $\widetilde{Z}=\C^s\times C$ is aspherical, hence $Z$ is as well.
\end{proof}

\bibliographystyle{plain}
\bibliography{Kahlerbib}

\end{document}